\documentclass[a4paper,12pt,reqno]{amsart}
\usepackage[left=3cm,right=3cm,top=3.5cm,bottom=3.5cm]{geometry}
\usepackage{amsmath,amssymb,amsthm,amscd,amsfonts}
\usepackage{latexsym}
\usepackage{graphicx}
\usepackage[all]{xy}
\usepackage{xcolor}
\usepackage[OT2,T1]{fontenc}
\usepackage{hyperref}
\usepackage{mathtools}
\usepackage{tikz}
\usetikzlibrary{arrows}
\usepackage{float}
\usepackage{caption}
\usepackage{multirow}
\usepackage{array}
\allowdisplaybreaks

\DeclareMathOperator{\rank}{rank}
\DeclareMathOperator{\Pic}{Pic}

\DeclareMathOperator{\Hom}{Hom}
\DeclareMathOperator{\Gal}{Gal}

\theoremstyle{plain}
\newtheorem{theorem}{Theorem}[section]
\newtheorem{lemma}[theorem]{Lemma}
\newtheorem{proposition}[theorem]{Proposition}
\newtheorem{corollary}[theorem]{Corollary}

\newtheorem{conjecture}[theorem]{Conjecture}

\numberwithin{equation}{section} \numberwithin{figure}{section}

\theoremstyle{definition}
\newtheorem{definition}[theorem]{Definition}
\newtheorem{example}[theorem]{Example}
\newtheorem{question}[theorem]{Question}

\theoremstyle{remark}
\newtheorem{remark}[theorem]{Remark}
\numberwithin{table}{section}

\newcommand{\dan}[1]{{\color{blue}{[Dan: #1]}}}

\newcommand{\Z}{{\mathbb Z}}
\renewcommand{\O}{{\mathcal O}}
\newcommand{\Q}{{\mathbb Q}}	
\newcommand{\R}{{\mathbb R}}
\newcommand{\F}{{\mathbb F}}
\newcommand\GG{\mathbb{G}}
\newcommand\Ga{\GG_\mathrm{a}}
\newcommand\Gm{\GG_\mathrm{m}}
\renewcommand{\P}{{\mathbb P}}	
\renewcommand{\H}{\mathrm{H}}
		
\renewcommand{\leq}{\leqslant}
\renewcommand{\geq}{\geqslant}

\begin{document}
\thispagestyle{empty}
	
\title{Fermat near misses and the integral Hilbert property}
\date{\today}

\author[J. Alessandrì]{Jessica Alessandr\`i}
\address{Max Planck Institute for Mathematics\\
Vivatsgasse 7\\
53111 Bonn\\
Germany}
\email{alessandri@mpim-bonn.mpg.de}
\urladdr{https://sites.google.com/view/jessicaalessandri}

\author[D. Loughran]{Daniel Loughran}
\address{Department of Mathematical Sciences \\
University of Bath \\
Claverton Down \
Bath\\ 
BA2 7AY\\
UK.}
\urladdr{https://sites.google.com/site/danielloughran/}


\renewcommand{\thefootnote}{\arabic{footnote}}
\setcounter{footnote}{0}

\begin{abstract}
    We consider the Diophantine equation $x^4 + y^4 - w^2 = n$ for $n \in \Z$, which is related to near misses for the quartic case of Fermat's Last Theorem. For certain $n$ we show that the set of solutions is infinite, or more generally \emph{not thin}. Our approach is via the geometry of del Pezzo surfaces of degree $2$, and we prove a more general result on non-thinness of integral points on double conic bundle surfaces.
\end{abstract}

\maketitle
\tableofcontents

\medskip
\section{Introduction}

\subsection{Fermat near misses}
Which integers $n$ are a sum of three integer cubes
$$x^3 + y^3 + z^3 = n \; ?$$
The case $n = 0$ is a special case of Fermat's Last Theorem; for more general $n$ one can view a solution as a ``Fermat near miss''. A conjecture of Heath--Brown \cite{HB92} predicts that the equation has a solution whenever $n \not \equiv 4,5 \bmod 9$. Moreover, if a solution exists, then he conjectured that there are infinitely many solutions. Despite much study of this problem, including recent important computational work finding large solutions for various integers $n$ \cite{Boo19, BS21}, this conjecture is wide open in general. Infinitude is only known when $n$ or $2n$ is a cube.

We consider analogues for more general families of equations:
$$F(x,y,z) = n$$
where $F$ is weighted homogeneous in $x,y,z$ with weights $p,q,r$, respectively. To have a similar structure to sums of three cubes we require that $p + q + r = \deg F$.
We consider the next simplest possible weights $(p,q,r) = (1,1,2)$, giving an equation of weighted degree $4$. We also consider diagonal equations, being indefinite to allow the possibility for infinitely many solutions. This leads us to consider
\begin{equation} \label{def:Fermat}
    x^4 + y^4 - w^2 = n, \quad n \in \Z.
\end{equation}
One can also view solutions to this equation as a kind of ``Fermat near miss'', as the equation $x^4 + y^4 = w^2$ plays a key role in the quartic case of Fermat's Last Theorem.

These equations \eqref{def:Fermat} have been recently been studied in \cite{Eliashar2023An=4}, where it was shown, rather surprisingly, that there are infinitely many solutions for $n = 8$ (the case $n = -1$ appears on Mathematics Stack Exchange and MathOverflow \cite{Tito, MO} as a challenge problem, currently open). The solutions for $n=8$ are constructed in \cite{Eliashar2023An=4} via a complicated recursive sequence. We put this argument into a geometric context, which allows us to prove existence of infinitely many solutions for a wide list of examples.
More than this,  we are able to show that the integral points are \emph{not thin} (see Definition \ref{def:IHP}); this implies for example that they are Zariski dense.

\begin{theorem} \label{thm:Fermat}
    Let $n \in \Z$ and consider the affine surface 
    $$U_n: \quad x^4 + y^4 - w^2 = n.$$   
    Then $U_n(\Z)$ is not thin provided there exists $m \in \Z$ and $(x,y,w) \in U_n(\Z)$ such that
    one of the following holds.
    \begin{enumerate}
        \item $n = m^4$ or $n = -4m^4$.
        \item $n = m^2$ is a square and $st(s+t)(s-t) > 0$, where
        $s = x^2-m, t = w - y^2$.
        \item $n = 2m^2$ is twice a square and $st\neq 0$ where $s = xy-m$, $t= x^2 - y^2-w$.
        \item $n = -2m^2$ is minus twice a square and $4s^4 - 12s^2t^2+t^4 >0$ and $st \neq 0$, where $s = xy-m$, $t= x^2 + y^2-w$.
    \end{enumerate}
\end{theorem}

Any solution for some $n$, yields a solution for $nq^4$ for any $q \in \Z$; so without loss of generality one can assume that $n$ is fourth power free. Theorem~\ref{thm:Fermat} gives indirect evidence for Heath--Brown's conjecture that any integer with at least one representation as a sum of three cubes has infinitely many. Theorem \ref{thm:Fermat} applies to some non-rational surfaces (see Remark \ref{rem:non_rational}), but the only case of sums of three cubes for which Zariski density is known is the Fermat cubic surface, which is rational.

\begin{table}[htb] 
    \centering
    \begin{tabular}{c|c|c|c|c|c|c} 
        $n$ & $1$ & $2$ & $-2$ & $-4$ & $8$  & $25$   \\  \hline
        $(x,y,w)$ & $(5,7,55)$ & $(15,33,1112)$ & $(47,39,2682)$ & $(2,2,6)$ & $(3,6,37)$ & $(5,5,35)$ 
    \end{tabular}
    \caption{Some solutions satisfying the hypotheses of Theorem~\ref{thm:Fermat}}
    \label{tab:solutions}
\end{table}

We do not know in general whether the equations \eqref{def:Fermat} have a solution, and there is no reason to expect this to be easier than sums of three cubes (some simple families do admit obvious solutions, e.g.~if $-n$ is a square). Hence a natural approach is computational. Our theorem says that provided one is able to find a ``good'' solution, then there are very many solutions. Good in the context of Theorem \ref{thm:Fermat} means a solution satisfying some real conditions; these real conditions describe unbounded domains, hence we expect very many such solutions if they exist.

\begin{remark} \label{rem:trivial}
    There are some ``trivial'' integer solutions for the cases $n = \pm 2m^2$. Let $xy = m$. Then we can rewrite the equation as $x^4 \mp 2x^2y^2 + y^2 = w^2$, which rearranges to $(x^2 \mp y^2)^2 = w^2$. Thus we obtain a solution by taking $w = x^2 \mp y^2$. These solutions are not ``good'' since they clearly satisfy $s= 0$ in the notation of Theorem \ref{thm:Fermat}. However, in the case $n = 2m^2$, Theorem \ref{thm:Fermat} applies as soon as one can find one non-trivial integer solution. 
\end{remark}

\subsection{Conic bundles}
We prove Theorem \ref{thm:Fermat} using geometric techniques. The natural compactifications of the surfaces $U_n$ lie in weighted projective space $\P(1,1,1,2)$, and define \emph{del Pezzo surfaces of degree $2$}. The arithmetic of del Pezzo surfaces is a rich very and ancient topic, with lower degree surfaces representing a greater challenge than higher degree surfaces. Our key observation is that the surfaces in Theorem \ref{thm:Fermat} admit a \emph{conic bundle structure}; this is not at all obvious from the equation and finding these structures and writing them down explicitly requires some detailed geometric arguments.  

A \textit{conic} on a del Pezzo surface $X$ is a connected reduced curve $C \subset X$ of arithmetic genus $0$ such that $-K_X \cdot C = 2$. A single conic moves in a pencil and gives rise to a conic bundle structure $\pi: X \to \P^1$ (see \cite[Lem.~2.8]{LS22}). Moreover if the degree of $X$ equals $1,2,4$, then there is actually a second associated conic bundle, which we call the \emph{dual} of $\pi$ (see Lemma \ref{lem:Str2conic}). For the surfaces in Theorem~\ref{thm:Fermat}, the dual is obtained by applying the Geiser involution $w \mapsto -w$. Given this, Theorem \ref{thm:Fermat} is then a special case of the following. 

\begin{theorem} \label{thm:dP}
    Let $X$ be a del Pezzo surface over $\Q$ of degree $d \in \{1,2,4\}$.
	Let $D \subset X$ be a smooth anticanonical divisor and $\mathcal{X}$ a projective model for $X$ over $\Z$ and $\mathcal{D}$ the closure of $D$ in $\mathcal{X}$. 

    Assume that $X$ admits a conic bundle structure $\pi_1: X \to \P^1$ with dual $\pi_2:X \to \P^1$.
    Assume that $(\mathcal{X} \setminus \mathcal{D})$ admits an integral point which lies on a smooth fibre $C$ of $\pi_1$ which meets $D$ in a real quadratic point such that $\pi_2|_D: D \to \P^1$ is unramified along $C \cap D$. Then $(\mathcal{X} \setminus \mathcal{D})(\Z)$ is not thin.
\end{theorem}

By definition $-K_X \cap C$ is a finite scheme of degree $2$; our hypotheses in Theorem~\ref{thm:dP} require this to be irreducible and have a real point. This is fairly natural to impose, otherwise the conic has only finitely integral points (see Lemma \ref{lem:pointsatinftyconic}). More philosophically, real conditions are frequently required in the study of integral points on varieties. Our condition implies that $D(\R) \neq \emptyset$; this condition is necessary for the result to hold. For example, for the equation
$$x^4 + y^4 + w^2 = n$$
there are clearly only finitely many integral solutions. In this case the relevant boundary divisor $x^4 + y^4 + w^2 = 0$ has no real point. The technical unramified condition is fairly mild as it holds for all but finitely many conics (see Lemma \ref{lem:meet_D_unramified}); in fact it is so rare that it is automatically satisfied in the situation of Theorem \ref{thm:Fermat} (see Lemma \ref{lem:Fermat}).

We have an analogous result for cubic surfaces. Recall that any conic bundle on a smooth cubic surface arises via taking the residual intersections of the family of planes through a fixed  line.

\begin{theorem} \label{thm:cubic}
    Let $X$ be a smooth cubic surface over $\Q$
	and $D \subset X$ be a smooth hyperplane section. Let $\mathcal{X}$ be a projective model for $X$ over $\Z$ and $\mathcal{D}$ the closure of $D$ in $\mathcal{X}$. 

    Assume that $X$ admits two lines and let $\pi_1,\pi_2$ denote the corresponding conic bundles. 
    Assume that $\mathcal{X} \setminus \mathcal{D}$ admits an integral point which lies on a smooth fibre $C$ of $\pi_1$ which meets $D$ in a real quadratic point such that $\pi_2|_D: D \to \P^1$ is unramified along $C \cap D$. Then $(\mathcal{X} \setminus \mathcal{D})(\Z)$ is Zariski dense.

    If moreover the lines are coplanar, then $(\mathcal{X} \setminus \mathcal{D})(\Z)$ is not thin.
\end{theorem}

This result allows us for instance to recover non-thinness of integral points for sums of three cubes when $n$ is a cube, a result originally due to Coccia \cite[Thm.~1.15]{Coccia2019TheSurfaces} (see Example \ref{ex:sums_three_cubes}). Theorems \ref{thm:dP} and \ref{thm:cubic} are all special cases of the following, which allows more general surfaces, applies over number fields, and to $S$-integral points. It is the main result of our article.

\begin{theorem}\label{thm:generalisation}
    Let $k$ be a number field and $X$ be a smooth projective surface over $k$ which
    admits two distinct conic bundle structures $\pi_i:X \to \P^1$ for $i = 1,2$. Let $D \subset X$ be a smooth irreducible divisor. Let $S$ be a finite set of places containing the archimedean places
    and $\mathcal{X}$ a projective model for $X$ over $\O_{k,S}$ and $\mathcal{D}$ the closure of $D$ in $\mathcal{X}$. Assume that the following conditions hold.
    \begin{enumerate}
        \item $X \setminus (D\cup E)$ is simply connected, where $E$ is the union of curves that are constant for both conic bundles. \label{hyp_intro:1}
        \item For all $P \in \P^1(k)$ we have $\pi_1^{-1}(P) \cdot D = \pi_2^{-1}(P) \cdot D = 2$. 
        \label{hyp_intro:3}         
        \item $\mathcal{X}\setminus \mathcal{D}$ admits an $\O_{k,S}$-integral point which lies on a smooth conic $C:=\pi_1^{-1}(P_0)$ with $P_0 \in \P^1(k)$ such that \label{hyp_intro:2}
        \begin{enumerate}
             \item $C^\circ(\O_{k,S})$ is infinite. \label{hyp_intro:2a}
             \item If $k = \Q$ or an imaginary quadratic extension of $\Q$ and $|S| = 1$, 
             then the set $\pi_2(C^\circ(\O_{k,S})) \cap \pi_2(D(k)) \subset \P^1(k)$ is finite. \label{hyp_intro:2b}
             \item The map $\pi_2|_D : D \to \P^1$ is unramified along $C \cap D$. \label{hyp_intro:2c}
		\end{enumerate} 
    \end{enumerate}
    Then $(\mathcal{X} \setminus \mathcal{D})(\O_{k,S})$ is not thin.
\end{theorem}

In the statement, $C^\circ(\O_{k,S})$ denotes the collection of integral points on $\mathcal{C} \setminus \mathcal{D}$, where $\mathcal{C}$ denotes the closure of $C$ in $\mathcal{X}$.
Let us discuss the assumptions in Theorem~\ref{thm:generalisation}. Condition \eqref{hyp_intro:1} is fundamental to the method as it guarantees that a finite cover of $X$ ramifies along a divisor which is horizontal with respect to one of the conic bundles. The condition $\pi_1^{-1}(P) \cdot D = 2$ in \eqref{hyp_intro:3} allows there to be conics in the family with infinitely many integral points. Condition \eqref{hyp_intro:2a} gives some integral points to work with so that we can generate more. The finiteness condition in \eqref{hyp_intro:2b} is fairly mild and holds for all but finitely many conics or if $D$ has positive genus (see Lemma~\ref{lem:conic_finiteness}). The condition \eqref{hyp_intro:2c} also holds for all but finitely many conics (see Lemma \ref{lem:meet_D_unramified}). For example, the assumption in Theorem \ref{thm:cubic} that the lines are coplanar is required to assure that Condition \eqref{hyp_intro:1} holds (see Lemma \ref{lem:cubic_lines} and Question \ref{ques:cubic}).

Theorem \ref{thm:Fermat} naturally suggests the following conjecture.

\begin{conjecture}
    Let $n \in \Z$ and assume that the equation
    $$x^4 + y^4 - w^2 = n$$
    has an integer solution. Then there are infinitely many solutions.
\end{conjecture}

An interesting case is when $-n$ is a square; here there are obvious solutions, but infinitude is not known in any case aside from when $n = -4m^4$ by Theorem~\ref{thm:Fermat}. Outside this case, there is no conic bundle structure by Proposition~\ref{prop:conic_bundle_Fermat}.

\begin{remark}
    The solutions to the case $n = 8$ of \eqref{def:Fermat} in \cite{Eliashar2023An=4} arise via a complicated recursive relation which is formally similar to the more classical recursive relation for solutions to the Pell equation: an affine conic. It was this observation which led us to search for a conic bundle on the surface, which ultimately led to the classification of conic bundle structures given in Proposition \ref{prop:conic_bundle_Fermat}. Thus the solutions found in \cite{Eliashar2023An=4} all come from integer points on finitely many conics. Our new innovation is to find infinitely many conics with integer solutions, in fact  so many conics to obtain the integral Hilbert property.

    Similarly, on Mathematics Stack Exchange \cite{Tito} is presented the relation
    $$(17p^2-12pq-13q^2)^4 + (17p^2+12pq-13q^2)^4 = (289p^4+14p^2q^2-239q^4)^2 + (17p^2-q^2)^4$$
    for $p, q \in \Z$, which is apparently due to E.~Fauquembergue from the late 1800's.
    Taking $17p^2-q^2 = 1$ one finds a Pell equation parametrising infinitely many solutions in the case $n = 1$. Proposition \ref{prop:conic_bundle_equation} shows that this Pell equation is in fact part of an infinite family of solutions coming from a conic bundle structure. It does not seem to have been previously known that the integral points are Zariski dense on this surface for $n =1$; this is a very special case of our results.
\end{remark}

\subsection{Method of proof and related literature}
Theorem \ref{thm:generalisation} is proved via the \textit{double fibration method}. This method was first introduced by Corvaja and Zannier \cite{Corvaja2017OnVarieties} to show that rational points on the Fermat quartic surface are not thin. This method was subsequently generalised by Demeio \cite{Demeio2020Non-rationalProperty,Demeio2021EllipticProperty} (elliptic surfaces) and Streeter \cite{Streeter2021HilbertVarieties} (conic bundles). Conic bundles and infinite automorphism groups are also used in \cite{KolVP, Kolllar} to prove Zariski density of integral points on certain affine cubic surfaces with singular boundary divisor.

The closest work to ours is Coccia \cite{Coccia2019TheSurfaces, Coccia2024ASurfaces}. Coccia proves in \cite[Thm.~1.7]{Coccia2024ASurfaces} a result on the integral Hilbert property for double conic bundle surfaces, with possibly singular boundary divisor. He uses this to obtain the integral Hilbert property \cite[Thm.~1.6]{Coccia2024ASurfaces} for del Pezzo surfaces after a possible finite field extension and after enlarging the set of places. A sketch of a proof is also given in \cite[\S6.4]{Coccia2024ASurfaces} regarding how to apply his method to a specific affine cubic surface over $\Z$. 
The assumptions in \cite[Thm.~1.7]{Coccia2024ASurfaces} imply that the fibres of the two conic fibrations have intersection number $2$, 
so this does not cover the case of del Pezzo surfaces of degree $2$ as in \eqref{def:Fermat}, where the intersection number is $4$.

Our main innovation is to make this method work with weaker assumptions, and with particular care constructing affine conics over $\Z$ with infinitely many integral points, since this is more difficult over $\Q$ as there is only one place at infinity (see Lemma \ref{lem:pointsatinftyconic}). We also require a result on strong approximation for integral points on conics (Proposition \ref{prop:v_adic_approximation}), which we combine with a result on lifting integral points to covers (Lemma \ref{lem:conic_finiteness}) and  topological arguments in Lemma~\ref{lem:Zariski_dense} to generate conics in the family with infinitely many integral points. We keep running our process to show that there are non-thin many conics with infinitely many integral points, from where our proof follows a similar strategy to \cite[Thm.~1.7]{Coccia2024ASurfaces}. 

In private communication, Coccia has informed us of forthcoming work where he applies the double fibration method to prove the integral Hilbert property for some affine cubic surfaces over $\Z$ with possibly singular boundary divisor, including important classical surfaces such as Markoff type cubic surfaces.  His approach uses the method sketched in \cite[\S6.4]{Coccia2024ASurfaces}, which is a similar strategy to our own.

\subsection*{Outline of the paper} In \S \ref{sec:background} we recall various background on conic bundle surfaces. In \S \ref{sec:double_conic} we prove our main result (Theorem \ref{thm:generalisation}), by building upon the method of Coccia \cite[Thm.~1.7]{Coccia2024ASurfaces}. In \S\ref{sec:del_Pezzo} we specialise our main theorem to del Pezzo surfaces and prove Theorems \ref{thm:dP} and \ref{thm:cubic}. In \S \ref{sec:Fermat} we give our application to Theorem \ref{thm:Fermat}. The key innovation here is making the non-obvious realisation that the surfaces in this theorem have a conic bundle structure and writing down the conic bundle structure explicitly; this is achieved in Propositions \ref{prop:conic_bundle_Fermat} and \ref{prop:conic_bundle_equation}.

\subsection*{Notation}
For a variety $X$ over a field $k$, say that $X$ is \textit{simply connected} if the \'etale fundamental group of the base change of $X$ to the algebraic closure is trivial.

Let $k$ be a number field and by $v$ a place of $k$.
For $S$ a finite set of places of $k$ containing the archimedean ones, we denote by 
$$\O_{k,S} = \{ \alpha \in k : v(\alpha) \geq 0 \text{ for all } v \notin S\}$$ the ring of $S$-integers of $k$.
For a non-archimedean $v$, we denote by $\O_v = \{ \alpha_v \in k_v : v(\alpha_v) \geq 0\}$ the associated valuation ring.

\subsection*{Funding}
This work was supported by UKRI Future Leaders Fellowship \texttt{MR/V021362/1}. JA thanks the Max Planck Institute for Mathematics in Bonn for their hospitality and financial support. JA is member of Italian ``National Group for Algebraic and Geometric Structures, and their Application'' (GNSAGA - INdAM).

\subsection*{Acknowledgements}
We are grateful to Julian Demeio and Harry Shaw for useful discussions, Sam Streeter for detailed comments which led to Lemma~\ref{lem:cubic_lines}. We thank Kumayl Vanat for providing the data in Table \ref{tab:solutions}, and  Ari Shnidman for making us aware of the paper \cite{Eliashar2023An=4} and discussing the geometry of the surface.  We also thank Simone Coccia for various comments and pointing out a mistake in an earlier version of Lemma 3.5.

\section{Preliminaries} \label{sec:background}
In this section we provide some background notions that we will need, in particular on del Pezzo surfaces and conic bundles.

\subsection{Integral Hilbert property}

We recall some basic facts about thin sets and the integral Hilbert property here. For additional background and context, see the recent survery article \cite{FehmJav2025}.

\begin{definition}
    Let $X$ be an integral variety over a field $k$. A \emph{thin map} is a morphism
    $Y \to X$ from an integral variety which admits no rational section. A \emph{thin subset}
    of $X(k)$ is a subset contained in the images of the rational points of
    the union of finitely many thin maps.
\end{definition}

Let $k$ now be a number field.

\begin{definition} \label{def:integral_model}
    Let $X$ be an integral variety over $k$ and $D$ a divisor on $X$. An \emph{integral model} over $\O_{k,S}$ for $(X,D)$ is a pair $(\mathcal{X}, \mathcal{D})$ of flat schemes over $\O_{k,S}$ together with a choice of isomorphism between $X$ and the generic fibre of $\mathcal{X}$, such that $\mathcal{D}$ is the closure of $D$ in $\mathcal{X}$.
    
    An \emph{$S$-integral point of $\mathcal X \setminus \mathcal D$} is a section $s_P$ of $\mathcal X \to \mathrm{Spec}(\O_{k,S})$ avoiding $\mathcal D$, i.e.\ for each $\mathfrak{p} \in \mathrm{Spec}(\O_{k,S})$ we have that $s_P(\mathfrak{p}) \notin \mathcal{D}_{\mathfrak{p}}$.
    We denote the set of these points by $(\mathcal X\setminus \mathcal{D})(\O_{k,S})$. If $D$ and the model is clear from the context, we also denote by $X^\circ = X \setminus D$ and $X^\circ(\O_{k,S}) =(\mathcal X\setminus \mathcal{D})(\O_{k,S})$.
\end{definition}

For a closed subset $C \subset X$, we obtain a natural model $\mathcal{C}$ for $C$ by taking the closure of $C$ in $\mathcal{X}$. We typically denote by $C^\circ :=C \setminus D$, and $C^\circ(\O_{k,S}) := (\mathcal{C} \setminus \mathcal{D})(\O_{k,S})$.

\begin{definition} \label{def:IHP}
    Let $\mathcal{X}$ be a finite type integral scheme over $\O_{k,S}$. We say that $\mathcal{X}$ satisfies the \emph{$S$-integral Hilbert property} if $\mathcal{X}(\O_{k,S})$ is not thin.
\end{definition}

\subsection{Integral points on curves}

First we recall the famous Siegel's Theorem on finiteness of integral points for curves of positive genus \cite{siegel2014einige}.
\begin{theorem}\label{thm:siegel}
    Let $C$ be an affine irreducible curve of genus $>0$ over $k$. Then for any integral model of $C$, the $S$-integral points of $C$ are finite.
\end{theorem}

We now take the following set up: let $C$ be a smooth projective curve of genus $0$ over $k$ and $D$  a reduced effective divisor on $C$ of degree at most $2$; we informally refer to $D$ as the ``points at infinity''. Let $S$ be a finite set of places of $k$ containing the archimedean ones and $\mathcal{C}$ a normal projective model for $C$ over $\O_{k,S}$, with $\mathcal{D}$ the closure of $D$ in $\mathcal{C}$. We write $C^\circ = C \setminus D$ and $C^\circ(\O_{k,S}) = (\mathcal{C} \setminus \mathcal{D})(\O_{k,S})$.

\begin{lemma}\label{lem:pointsatinftyconic}
    If $C^\circ(\O_{k,S}) \neq \emptyset$, then $C^\circ(\O_{k,S})$ is infinite if and only if 
    there is a place $v \in S$ for which $D(k_v) \neq \emptyset$, and in the case that $\#D(k) = 2$ we have $|S| \geq 2$.
\end{lemma}
\begin{proof}
    This is a geometric reformulation of a well-known criterion, see e.g.~\cite[Thm.~1.1]{Alvanos2009CharacterizingPoints} or \cite[Cor.~5.4]{Hassett2001}. One sees this
    by noting that if $D$ is irreducible, then $D(k_v) \neq \emptyset$ if and only if $v$ is completely split in the field extension $k \subset k(D)$.
\end{proof}

Another result \cite[Lem.~5.3]{Coccia2024ASurfaces} we require shows that only finitely many integral points lift to  certain covers, as a consequence of Siegel's Theorem.

\begin{lemma}\label{lemma:5.3}
    Assume that $\deg D = 2$. Let $V$ be a smooth projective irreducible curve and $f: V \to C$ be a finite map. Assume that one of the following holds:
    \begin{enumerate}
        \item the map $f$ is not totally ramified on $D$; or
        \item the map $f$ has at least one branch point outside the support of $D$.
    \end{enumerate}
    Then $f(V(k)) \cap C^\circ(\O_{k,S})$ is finite.
\end{lemma}

It is well-known that any non-simply connected variety fails strong approximation; this is the case for any affine conic with two geometric points at infinity. Nonetheless we shall require the following special case of strong approximation, which says that integral points can be arbitrarily close to the points at infinity.

\begin{proposition} \label{prop:v_adic_approximation}
    Assume that $C^\circ(\O_{k,S})$ is infinite. Let $v \in S$ be such that $D(k_v) \neq \emptyset$. Let $W \subset C(k_v)$ be an open neighbourhood of $D(k_v)$. Then
    $C^\circ(\O_{k,S}) \cap W$ is infinite.
\end{proposition}
\begin{proof}
    We start with the perspective from \cite[\S~5.1, 5.2]{Hassett2001}. By \emph{loc.~cit.} there is an action of a $1$-dimensional linear algebraic group $G$ on $C^\circ$ which realises $C^\circ$ as a $G$-torsor, and this torsor action extends to a group scheme model $\mathcal{G}$ of $G$. On choosing an integral point on $C^\circ$, it suffices to prove the analogous result for $\mathcal{G}$, on identifying $C^\circ = G$, since the action of  $\mathcal{G}(\O_{k,S})$ on the fixed integral point allows us to view $\mathcal{G}(\O_{k,S}) \subset \mathcal{C}^\circ(\O_{k,S})$. Moreover as explained in \emph{loc.~cit.}, changing the model $\mathcal{G}$ potentially changes $\mathcal{G}(\O_{k,S})$ only up to a finite index subgroup. This will not affect our proof method, hence we may choose the model $\mathcal{G}$. It is also easy to see under our assumptions that there is an element $g \in \mathcal{G}(\O_{k,S})$ of infinite order (this follows from similar arguments to below).

    If $v$ is archimedean then the subgroup generated by $g$ is isomorphic to $\Z$ and it 
    embeds as a discrete subgroup of $G(k_v)$. However $C(k_v)$ is compact, thus $C(k_v) \setminus W$ is also compact. It follows that $\Z \setminus W$ is finite, thus $\Z \cap W$ is infinite, as required.

    If $v$ is non-archimedean we need to be more careful as the subgroup generated by a point need not be discrete (consider $\Z \subset \Z_p)$. Assume first that $\deg D = 1$; here we may take $\mathcal{G} = \Ga$. Then $\Ga(\O_{k,S}) = \O_{k,S}$ is dense in $k_v$, so the result is clear.

    Now assume that $\deg D = 2$, so that $G$ is a $1$-dimensional algebraic torus.     
    We use some properties of algebraic tori over local fields, which we now recall. Let $T$ be an algebraic torus over $k$ and $v$ a non-archimedean place of $k$ which is unramified in the splitting field of $T$. Let $T(\O_v)$ be the maximal compact subgroup of $T(k_v)$. Denote by $\widehat{T} = \Hom(T,\Gm)$ the character lattice of $T$. This is a Galois module which is finitely generated and free as an abelian group. Then the \emph{degree map}
    $$\deg_v: T(k_v) \to \Hom(\widehat{T}^{\Gamma_{v}},\Z), \quad t_v \mapsto \left( \, \chi_v \mapsto v(\chi_v(t_v)) \, \right)$$
    is well-defined and continuous where $\Gamma_{v} := \Gal(\bar{k}_v/k_v)$. By \cite[Lem.~3.9]{Bou11} it induces a short exact sequence
    \begin{equation} \label{seq:Shyr}
    0 \to T(\O_v) \to T(k_v) \to \Hom(\widehat{T}^{\Gamma_{v}}, \Z) \to 0.
    \end{equation}
    We use this as follows. Our torus $G$ is either $\Gm$ or the norm one torus $\mathrm{R}_{k(D)/k}^1 \Gm$, corresponding to a trivial or non-trivial Galois action on the character lattice, respectively. Moreover, in the latter case, as $D(k_v) \neq \emptyset$, the place $v$ is completely split in $k(D)$, so that the action of $\Gamma_{v}$ on the character lattice is trivial. In both cases, a theorem of Shyr \cite{Shy77} (a version of Dirichlet's theorem for algebraic tori) implies that $\mathcal{G}(\O_{k,S})/ \mathcal{G}(\O_{k,S \setminus \{v\} }) \cong \Z$. As $\mathcal{G}(\O_v)$ is a finite index subgroup of the maximal compact subgroup $G(\O_v)$, we deduce there is an element $g \in \mathcal{G}(\O_{k,S})$ such that $g \notin G(\O_v)$. In which case by \eqref{seq:Shyr} the degree $\deg_v(g)$ is non-trivial. It easily follows that $g$ generates an infinite discrete subgroup of $G(k_v)$, and then the same argument as in the archimedean case applies to complete the proof.
\end{proof}

\subsection{Del Pezzo surfaces} 
\begin{definition}
     A \emph{del Pezzo surface} over a field $k$ is a smooth, geometrically irreducible, and proper surface $X$ such that the anticanonical class $-K_X$ is ample. The \emph{degree} $d$ of $X$ is $K_X^2$; it is an integer in the range $1 \leq d \leq 9$. A \emph{line} on $X$ is a geometrically integral curve $L$ with $L\cdot K_X = -1$ and $L\cdot L = -1$.
\end{definition}

Of particular interest to us are del Pezzo surfaces of degree $2$. These are surfaces of degree $4$ in the weighted projective space $\P(1,1,1,2)$ of the form
$$X: w^2 = F(x,y,z)$$
where $\deg F = 4$, provided the characteristic of $k$ is not equal to $2$. The anticanonical map $X \to \P^2$ is given by projecting to $(x,y,z)$, and realises the surface as a double cover of $\P^2$ ramified in the smooth quartic curve $F(x,y,z) = 0$. This induces the \emph{Geiser involution} $\iota: X \to X$ which swaps the two sheets of the cover.

\subsection{Conic bundles}

\begin{definition}
    Let $X$ be a smooth projective surface $X$ over a field $k$. A \emph{conic bundle} is a morphism $\pi: X \to \mathbb P^1$ whose fibres are isomorphic to plane conics.
\end{definition}

The following facts are well-known and follow from the adjunction formula.
\begin{lemma} \label{lem:conic_bundle_facts}
	Let $\pi:X \to \P^1$ be a conic bundle surface with fibre class $C$. Then
	\begin{enumerate}
        \item $-K_X \cdot C = 2$.
        \item The singular fibers over $\bar{k}$ have the form $L_1+L_2$, with $L_i^2 = -1$ and $L_1\cdot L_2=1$, i.e.~they are isomorphic over $\bar{k}$ to two lines meeting in a single point.        
	 \end{enumerate}
\end{lemma}

To prove the existence of a conic bundle structure, we will use the following.

\begin{lemma} \label{lem:conic_bundle_exists}
Let $X$ be a del Pezzo surface over a field $k$ of characteristic $0$ with a rational point. Assume that there exists  a Galois invariant collection $\mathfrak{L} \subset X_{\bar{k}}$ of lines which are disjoint collections of pairs of lines meeting in a single point. Then $X$ admits a conic bundle structure with $\mathfrak{L}$ forming a subset of the singular fibres.
\end{lemma}
\begin{proof}
    This is a special case of \cite[Prop.~5.2]{Frei2018RationalSurfaces}.
    We sketch a proof, since it also furnishes an effective method to construct a conic bundle structure explicitly. Moreover the Stein factorisation step at the end of \emph{loc.~cit.} is not actually required. 
    
    Assume that $\mathfrak{L}$ consists of $2n$ lines. A cohomological calculation shows that $\mathfrak{L} \in \Pic X_{\bar{k}}$ is the class of $n$ singular fibres of a conic bundle, i.e~$[\mathfrak{L}] = nC$ where $C \in \Pic X_{\bar{k}}$ is the class of the fibre of a conic bundle. Since $\Pic X_{\bar{k}}$ is torsion free we see that $C$ is Galois invariant; as $X(k) \neq \emptyset$ the class $C$ thus descends to $k$.
    
    One achieves an explicit Galois descent and construction of the map as follows. We have $\dim \H^0(X,\mathfrak{L}) = 2^n$. The image of the induced map $X \to \P^{2^n-1}$ is a curve of genus $0$. As $X(k) \neq \emptyset$ this curve has a rational point. Thus choosing an isomorphism of this curve with $\P^1$ yields a map $X \to \P^1$ which is the required conic bundle.
\end{proof}

\section{Double conic bundle surfaces} \label{sec:double_conic}

The aim of this section is to prove Theorem \ref{thm:generalisation}.

\subsection{Fundamental groups}
We begin with some results on fundamental groups of complements of divisors and blow-ups.
Let $k$ be a field.

In the following $\pi_1$ denotes the \'etale fundamental group, which should not be confused with a conic bundle morphism (this conflict will occur nowhere else).

\begin{lemma} \label{lem:blow_down_fundamental group}
    Let $X$ be a smooth proper surface over $k$ and $D \subset X$ an effective divisor. Let $E \subset X$ be a collection of pairwise disjoint $(-1)$-curves and $f: X \to X'$ the blow down of $E$. Then $\pi_1(X \setminus (D \cup E)) \cong \pi_1( X' \setminus f(D))$.
\end{lemma}
\begin{proof}
    The map $f$ induces an isomorphism $X \setminus (D \cup E) \cong X' \setminus f(D \cup E)$. Thus it suffices to show that $\pi_1( X' \setminus f(D)) = \pi_1( X' \setminus f(D \cup E))$.
    Let $L \subset E$ be an irreducible component. If $L \cap D \neq \emptyset$ then $f(D \cup L) = f(D)$ so there is nothing to prove. If $L \cap D = \emptyset$, then $f(D \cup L) = f(D) \sqcup f(L)$ is a disjoint union. The result then follows from the fact that removing a point from a smooth surface does not change the fundamental group, by Zariski purity of the branch locus.
\end{proof}

\begin{lemma} \label{lem:not_simply_connected}
    Let $X$ be a smooth proper surface over $k$ which admits two conic bundle
    structures $\pi_i: X \to \P^1$ with fibre classes $C_1,C_2$, respectively, such that $C_1 \cdot C_2 = 1$. Let $E$ be the divisor of curves which are constant with respect to both fibrations and $D \subset X$ an integral divisor with $C_1 \cdot D = C_2 \cdot D = 2$. Then $X \setminus (D \cup E)$ is not simply connected.
\end{lemma}
\begin{proof}
    We may assume that $k$ is algebraically closed. Blowing down fibre components of $\pi_1$, we obtain a map $f:X \to X'$ to a smooth projective surface which admits a conic bundle with smooth fibres, i.e.~a ruling. We now carefully analyse the possibilities for $X'$.

    Recall that the ruled surfaces are exactly the Hirzebruch surfaces $\F_n$ for $n \geq 0$. Denote by $F$ the fibre class of the ruling and $E$ the zero section. These classes generate the effective cone of $\F_n$, and we have the intersection numbers
    $$E^2 = -n, \quad F^2 = 0, \quad E \cdot F = 1.$$
    In our setting we are given integral curves $C,D \subset \F_n$ such that
    $$F \cdot C = 1, \quad F \cdot D = C \cdot D = 2.$$
    We analyse the possibilities for these curves. Write $C = e_1E + f_1F$ and $D = e_2E + f_2F$ for $e_i,f_i \geq 0$ integers. From $F \cdot C = 1$ and $F \cdot D = 2$ we obtain $e_1 = 1$ and $e_2 = 2$. From $C \cdot D = 2$ a short calculation yields $2f_1 + f_2 = 2n + 2$. As $D$ is integral we must have $f_2 > 0$, since $2E$ is not integral. Moreover, clearly $f_2$ is even. It follows that $f_1 \leq n$. However for $n \geq 1$, 
    the linear system $|C|=|E + f_1 F|$ contains no integral elements, provided that $0 \leq f_1 < n$. Indeed we have $E \cdot (E + f_1F) = -n + f < 0$, thus $E$ is a fixed component of the linear system. We deduce that $f_1 = n$, hence $C = E + nF$ and $D = 2E + 2F$. But then $E \cdot D = 2 - 2n$, so that $E$ is a fixed component of $|D|$ providing $n \geq 2$, which contradicts that $D$ is integral.
    
    Returning to our original proof, we conclude that $n = 0,1$, i.e.~that $X' = \P^1 \times \P^1$ with the conic bundles inducing the two rulings, or $X'$ is blow-up of $\P^2$ in a single point with one conic bundle corresponding to the ruling and the other coming from pulling-back a pencil of lines in $\P^2$ not passing through the blown-up point.
    Applying Lemma \ref{lem:blow_down_fundamental group} we find that the fundamental group of $X \setminus (D \cup E)$ is the same as $X' \setminus f(D)$. However in both cases $f(D)$ has class $2E + 2F$. Then the complement is not simply connected, since the Picard group has non-trivial $2$-torsion, as required.
\end{proof}

\subsection{Preparations}
We now take the set up of Theorem \ref{thm:generalisation}. We recall our hypotheses:
    \begin{enumerate}
        \item $X \setminus (D\cup E)$ is simply connected, where $E$ is the union of curves that are constant for both conic bundles. \label{hyp:1}
        \item For all $P \in \P^1(k)$ we have $\pi_1^{-1}(P) \cdot D = \pi_2^{-1}(P) \cdot D = 2$.
        \label{hyp:3}         
        \item $\mathcal{X}\setminus \mathcal{D}$ admits an $\O_{k,S}$-integral point which lies on a smooth conic $C:=\pi_1^{-1}(P_0)$ such that \label{hyp:2}
        \begin{enumerate}
             \item $C^\circ(\O_{k,S})$ is infinite. \label{hyp:2a}
             \item If $k = \Q$ or an imaginary quadratic extension of $\Q$ and $|S| = 1$, 
             then the set $\pi_2(C^\circ(\O_{k,S})) \cap \pi_2(D(k)) \subset \P^1(k)$ is finite. \label{hyp:2b}
            \item The map $\pi_2|_D : D \to \P^1$ is unramified along $C \cap D$. \label{hyp:2c}
		\end{enumerate} 
    \end{enumerate}
For $P \in \mathbb P^1$,  we denote by $C_{1,P}$ and $C_{2,P}$ the fibre above $P$ with respect to the fibration $\pi_1$ and $\pi_2$, respectively. Recall our notation and conventions for integral points and models from Definition \ref{def:integral_model}.

The fact that the conic bundles in Theorem \ref{thm:generalisation} are distinct means exactly that $C_{1,P} \cdot C_{2,P} > 0$. However, given hypothesis \eqref{hyp:1} on simply connectedness of $X\setminus(D \cup E)$, by Lemma \ref{lem:not_simply_connected} we may assume that actually
\begin{equation} \label{eqn:not_a_section}
    C_{1,P} \cdot C_{2,P} > 1.
\end{equation}
Despite this improvement looking relatively minor, actually it is fundamental to our method. It is used to prove a key finiteness result (Lemma \ref{lem:conic_finiteness}), which is used multiple times in our proof.

We first consider the assumptions in \eqref{hyp:2}.
For \eqref{hyp:2a}, by Lemma~\ref{lem:pointsatinftyconic} there is a simple criterion for whether a given conic with an integral point admits infinitely many. Condition \eqref{hyp:2b} is fairly mild as the following shows. In the following statements the same results naturally hold if the roles of $\pi_1$ and $\pi_2$ are swapped.

\begin{lemma} \label{lem:disjoint_branch_locus}
    Let $\Omega \subset \P^1$ be a finite collection of closed points. Then for all but finitely many $P \in \P^1(k)$, the branch locus of the map $\pi_2: C_{1,P} \to \P^1$, induced by $\pi_2$, is disjoint from $\Omega$.
\end{lemma}
\begin{proof}
    If $\pi_2: C_{1,P} \to \P^1$ ramifies over a point $Q \in \P^1$, then $C_{1,P} \cap C_{2,Q}$ is non-reduced. However, by Lemma \ref{lem:conic_bundle_facts}, the conic $C_{2,Q} \subset X$ is reduced. Therefore, as $k$ has characteristic $0$, the restriction $\pi_1: C_{2,Q} \to \P^1$ is generically smooth, which implies that $C_{1,P} \cap C_{2,Q}$ is non-reduced for only finitely many $P$. The result now easily follows.
\end{proof}

\begin{lemma} \label{lem:conic_finiteness}
    Let $f: Y \to \P^1$ be a finite
    morphism with $Y$ a smooth projective curve and $\deg f > 1$. 
    Let $P \in \P^1(k)$ and $C:=\pi_1^{-1}(P)$. 
	The set $\pi_2(C^\circ(\O_{k,S})) \cap f(Y(k)) \subset \P^1(k)$ is finite if
	\begin{enumerate}
		\item the genus of $Y$ is at least $1$, or
		\item the scheme $C \cap D$ is reduced 
        and the maps $f: Y \to \P^1$ and $\pi_2|_C: C \to \P^1$ have disjoint branch loci.
	\end{enumerate}
	Moreover, this finiteness holds for all but finitely many $P \in \P^1(k)$.
\end{lemma}
\begin{proof}
    Let $\mathcal{Y} \to \P^1$ be a model for $Y$.
    Consider the fibre product 
    \begin{figure}[H]
       \centering
       \begin{tikzpicture}
          \node (a) at (-1,2) {$\mathcal{F} = \mathcal{C} \times_{\P^1} \mathcal{Y}$};
            \node (b) at (2,2) {$\mathcal{Y}$};
            \node (c) at (-1,0) {$\mathcal{C}$};
            \node (d) at (2,0) {$\mathbb P^1$.};

            \node at (0.5,2.3) {$p$};
            \node at (-0.7,1) {$q$};
            \node at (2.3,1) {$f$};
            \node at (0.5,0.3) {$\pi_2|_C$};
        
            \draw[-stealth] (a) --(b);
              \draw[-stealth] (a) --(c);
            \draw[-stealth] (c) --(d);
              \draw[-stealth] (b) --(d);
        \end{tikzpicture}
    \end{figure}
    Here $F:=\mathcal{F}_k$ is a curve, though possibly reducible. The finiteness in the statement
    is equivalent to the finiteness of $(\mathcal{F}\setminus q^{-1}(\mathcal{C}\cap \mathcal{D}))(\O_{k,S})$.

    Assume that (1) holds. Since $Y$ has genus at least $1$, then so does every irreducible component of $F$. 
    The required finiteness now follows from Siegel's Theorem  (Theorem \ref{thm:siegel}).

    Now assume (2). Since the branch loci are disjoint, $F$ is smooth and geometrically irreducible
    by \cite[Lem.~2.8]{Streeter2021HilbertVarieties}.  Note that as the intersection product between $C$ and a fibre of $\pi_2$ is greater than $1$ by \eqref{eqn:not_a_section}, we see that $\deg(\pi_2 |_C) \geq 2$. 
    As $f$ has at least $2$ branch points over $\bar{k}$, we deduce that $q$ has at least $4$ branch points over $\bar{k}$. This follows from the compatibility of ramification indices with composition of morphism and the fact that $f$ and $q$ have disjoint branch loci. Since $C \cap D$ is reduced of degree $2$, we find that $q$ has at least one branch point outside of the support of $D$. The required finiteness then follows from Lemma \ref{lemma:5.3}.
  
    The final part follows from Lemma \ref{lem:disjoint_branch_locus}, and a similar argument which shows that $C \cap D$ is non-reduced for only finitely many fibres, since $D$ is reduced.
\end{proof}

Condition \eqref{hyp:2c} is also rather mild, as the following shows.

\begin{lemma} \label{lem:meet_D_unramified}
    Let $\Omega \subset D$ be a finite collection of closed points. Then for all but finitely many $P \in \P^1(k)$ we have $C_{1,P} \cap \Omega = \emptyset$. In particular, for all but finitely many $P \in \P^1(k)$ the map $\pi_2|_D : D \to \P^1$ is unramified along $C_{1,P} \cap D$.
\end{lemma}
\begin{proof}
    Immediate from the fact that the linear system given by the conics is base-point free and that $\pi_2|_D : D \to \P^1$ ramifies in only finitely many points.   
\end{proof}

We next show a strong version of Zariski density of $S$-integral points. By \eqref{hyp:2a} there is a conic on the surface with infinitely many integral points. The main difficulty is finding conics in the second fibration passing through these integral points which also have infinitely many integral points. This is obtained via the criterion in Lemma \ref{lem:pointsatinftyconic} and a topological argument using Proposition~\ref{prop:v_adic_approximation}, which is a substitute for strong approximation in our setting, as well as an additional argument over $\Q$ or imaginary quadratic fields using Lemma \ref{lem:conic_finiteness} to avoid very special cases. 
Let
\[ A_i = \{ P \in \mathbb P^1(k) \mid C_{i,P}^\circ(\mathcal O_S) \text{ is infinite} \}, \quad \text{for } i = 1,2.\]
Our strong version of Zariski density is as follows.

\begin{lemma} \label{lem:Zariski_dense}
    The sets $A_1$ and $A_2$ are infinite. Moreover, for all but finitely many $P \in A_1$, there are infinitely many  $x \in C_{1,P}^\circ(\O_{k,S})$ such that $C^\circ_{2,\pi_2(x)}(\O_{k,S})$ is infinite.
\end{lemma}
\begin{proof}
Let $C$ be a smooth conic as in hypothesis \eqref{hyp:2} of Theorem \ref{thm:generalisation}. By Lemma~\ref{lem:pointsatinftyconic}, there is a place $v \in S$ such that $(C \cap D)(k_v) \neq \emptyset$. Let $W_D \subseteq D(k_v)$ be the complement of the ramification points of $\pi_2|_D : D \to \P^1$; this is open in $D(k_v)$ as $D$ is smooth. Hence $\pi_2(W_D) \subset \P^1(k_v)$ is open as an \'etale morphism is open for the $v$-adic topology. We now consider the cylinder $W := \pi_2^{-1}(\pi_2(W_D)) \subset X(k_v)$, which is an open subset of $X(k_v)$. For any $P \in \pi_2(W)$, since $\pi_2(W) = \pi_2(W_D)$ we have $(C_{2,P} \cap D)(k_v) \neq \emptyset$.

Since $(C \cap D)(k_v) \subseteq W_D$ by hypothesis \eqref{hyp_intro:2c}, from Proposition \ref{prop:v_adic_approximation} we see that $W$ contains infinitely many $\O_{k,S}$-points of $C^\circ$. So let $x \in C^\circ(\O_{k,S}) \cap W$ and let $C_{2,\pi_2(x)}:= \pi_2^{-1}(\pi_2(x))$ be the conic from the second fibration passing through $x$. Then $\pi_2(x) \in \pi_2(W)$ and hence $(C_{2,\pi_2(x)} \cap D)(k_v) \neq \emptyset$. As $C_{2,\pi_2(x)}^\circ(\O_{k,S}) \neq \emptyset$ by construction, we deduce from Lemma~\ref{lem:pointsatinftyconic} that $C_{2,\pi_2(x)}^\circ(\O_{k,S})$ is infinite, unless we are in the special case that $\#(C_{2,\pi_2(x)} \cap D)(k) = 2$ and $|S| = 1$. In this special case we have assumed in hypothesis \eqref{hyp:2b} that $\pi_2(C^\circ(\O_{k,S})) \cap \pi_2(D(k))$ is finite, which implies that the boundary has two rational points for only finitely many choices of $x$. Hence there are infinitely many $x \in C_1^\circ(\O_{k,S})$ such that  $C_{2,\pi_2(x)}^\circ(\O_{k,S})$ is infinite. 
We complete the proof by repeating this process with $\pi_1$ and $\pi_2$ swapped, noting that only finitely many conics for $\pi_2$ fail our conditions by Lemmas \ref{lem:conic_finiteness} and  \ref{lem:meet_D_unramified}.
\end{proof}

The next step is showing that the sets $A_i$ are not thin.

\begin{lemma}\label{lemma:finite_set}
    Let $T \subset \P^1(k)$ be thin. For all but finitely many $P \in A_1$, the set $\pi_2(C^\circ_{1,P}(\mathcal O_{k,S})) \cap T$ is finite.
\end{lemma}
\begin{proof}
    As $T$ is thin, there is a finite collection $\phi_i: Y_i \to \mathbb P^1$, $i = 1, \ldots, r$, with $Y_i$ a smooth geometrically irreducible curve and $\deg(\phi_i)>1$, such that $T \subseteq \bigcup_{i=1}^r \phi_i(Y_i(k))$. The result then follows from applying Lemma \ref{lem:conic_finiteness} to each $\phi_i$.
\end{proof}

Naturally Lemma \ref{lemma:finite_set} holds if we swap the role of $\pi_1$ and $\pi_2$.

\begin{proposition}\label{prop:A_not_thin}
    The sets $A_1$ and $A_2$ are not thin.
\end{proposition}
\begin{proof}
    We do it for $A_2$, the other case being analogous. Assume that $A_2$ is thin. Then by Lemma \ref{lemma:finite_set}, for all but finitely many $P \in A_1$, the set $\pi_2(C^\circ_{1,P}(\mathcal O_S)) \cap A_2$ is finite. 
    However this contradicts the second part of  Lemma \ref{lem:Zariski_dense},
    as required.
\end{proof}

\subsection{Proof of Theorem \ref{thm:generalisation}]}
Given the above results, the remaining steps are quite similar to the proof of \cite[Thm.~1.7]{Coccia2024ASurfaces}, so we shall be slightly brief. We will use the following definition by Coccia \cite[Def.~5.2]{Coccia2024ASurfaces}. Let $\pi= \pi_j$, with $j \in \{1,2\}$.
\begin{definition}
    We say that a cover $\phi: Y \to X$ is \emph{$\pi$--unramified} if:
    \begin{itemize}
        \item the fiber $Y(P):= \phi^{-1}(\pi^{-1}(P))$ is geometrically irreducible for all but finitely many $P \in \P^1(k)$.
        \item the ramification locus of $\phi$ is contained in the union of $D$ and finitely many fibres of $\pi$.
    \end{itemize}
    Otherwise, we say that the cover is \emph{$\pi$--ramified}.
\end{definition}

We argue by contradiction. Assume that there exist finitely many finite maps $\phi_i : Y_i \to X$ of degree $>1$, with $Y_i$ a geometrically irreducible normal projective variety, such that $X^\circ(\O_{k,S}) \setminus \cup_i (\phi_i(Y_i(k))$ is contained in a curve $Z$. By Zariski's purity theorem the branch locus $B_i$ of $\phi_i$ is a divisor.

No cover $\phi_i$ can be both $\pi_1$--unramified and $\pi_2$--unramified.
    Indeed, in this case we would have that $B_i$ is contained in $D \cup E$, where $E$ is the union of curves on which both $\pi_1$ and $\pi_2$ are constant. By hypothesis \eqref{hyp:1} of Theorem \ref{thm:generalisation}, $X\setminus (D \cup E)$ is simply connected, thus $\phi_i$ would be trivial.
    
    As in \cite[Thm.~1.7, Claim 1]{Coccia2024ASurfaces}, we have that there exists a thin set $T_1 \subseteq\mathbb P^1(k)$ such that if $P \notin T_1$ then only finitely many points of $C^\circ_{1,P}(\O_{k,S})$ lift to the covers $\phi_i$ that are $\pi_1$--ramified.
    If all covers were $\pi_1$--ramified then, by Proposition \ref{prop:A_not_thin}, for all $P \in A_1 \setminus T_1$ only finitely many points of $C^\circ_{1,P}(\O_{k,S})$ lift, i.e.~we have infinitely many points on $C^\circ_{1,P}(\O_{k,S})$ that are \emph{not} in the image of the $\phi_i$. Hence we have a contradiction.
    Therefore we assume that there exists at least one $\pi_1$--unramified cover. Again as in \cite[Thm.~1.7, Claim 2]{Coccia2024ASurfaces}, 
    we find a thin set $T_1 \subseteq\mathbb P^1( k)$ (obtained enlarging $T_1$) such that, for $Q \notin T_1$, \emph{all} points of $C^\circ_{1,Q}(\O_{k,S})$ are lifted to $\pi_1$--unramified covers.

    As in the first case ($\pi_1$--ramified covers), we find a thin set $T_2$ such that for each $P \in A_2 \setminus T_2$ only finitely many points of $C^\circ_{2,P}(\O_{k,S})$ lift to $\pi_1$-unramified covers (which are $\pi_2$-ramified).

    Since $T_1$ is thin, by Lemma \ref{lemma:finite_set} applied swapping the role of $\pi_1$ and $\pi_2$, for all but finitely many $P \in A_2$, we have that $T_1 \cap \pi_1(C^\circ_{2,P}(\O_{k,S}))$ is finite, i.e.~only finitely many $x \in C^\circ_{2,P}(\O_{k,S})$ are such that $\pi_1(x) \in T_1$.
    Thus we can choose $P$ in $A_2 \setminus T_2$ such that there are infinitely many $x \in C^\circ_{2,P}(\O_{k,S})$ with $\pi_1(x) \notin T_1$ and $C^\circ_{1,\pi_1(x)}(\O_{k,S})$ is infinite by Lemma~\ref{lem:Zariski_dense}. Since $P \notin T_2$, then only a finite number of points of $C^\circ_{2,P}(\O_{k,S})$ are lifted to rational points of $\pi_1$-unramified covers.
    On the other hand, $\pi_1(x) \notin T_1$, so \emph{all} points 
    of $C^\circ_{1,\pi(P)}(\O_{k,S})$ are lifted to $\pi_1$-unramified covers. This gives a contradiction and hence finishes the proof. \qed 

\begin{example} \label{ex:bad_conic_bundle}
	Not all conic bundles with Zariski dense set of integral points satisfy the conclusion of Lemma \ref{lem:Zariski_dense}. Consider the affine surface $U: xy = t$ over $\Z$. This is isomorphic to $\mathbb{A}^2_{\Z}$, hence $U(\Z)$ is not thin. It is equipped with various conic bundles, e.g.~the fibration $f: U \to \mathbb{A}^1, (x,y,t) \to t$. But each fibre of $f$ has only finitely many integral points. This shows that our assumptions are necessary for our method (here the boundary divisor is reducible).
\end{example}

\section{Del Pezzo surfaces} \label{sec:del_Pezzo}
In this section we make explicit the applications of Theorem \ref{thm:generalisation} to del Pezzo surfaces, including the proofs of Theorems \ref{thm:dP} and \ref{thm:cubic}.

\subsection{Double conic bundles}
\begin{lemma}\label{lem:Str2conic}
    Let $X$ be a del Pezzo surface of degree $d \in \{1,2,4\}$ with a conic bundle with fibre class $C \in \Pic X$. Then there exists a second conic bundle with fibre class $C' \in \Pic X$ such that $-(4/d)K_X = C + C'$.
\end{lemma}
\begin{proof}
	See \cite[Lem.~3.5]{Streeter2021HilbertVarieties} and its proof.
\end{proof}
We call the second conic bundle from Lemma \ref{lem:Str2conic} \emph{the dual} of the original conic bundle.

\begin{corollary}\label{cor:dP}
    Let $k$ be a number field and $X$ be a smooth del Pezzo surface over $k$ of degree $d \leq 7$.
	Let $D \subset X$ be a smooth  anticanonical divisor.  Let $S$ be a finite set of places containing the archimedean places and $(\mathcal{X}, \mathcal{D})$ a projective model for $(X,D)$ over $\O_{k,S}$.
    Assume that either
    \begin{enumerate}
        \item[(1a)] $X$ admits two distinct conic bundle structures $\pi_1,\pi_2$ such that $X \setminus (D\cup E)$ is simply connected, 
        where $E$ denotes the union of those curves that are constant with respect to both conic bundles, or 
    	\item[(1b)] $d \in \{1,2,4\}$ and $X$ admits a conic bundle structure $\pi_1$ with dual $\pi_2$,
	\end{enumerate}
     and that
     \begin{enumerate}
    \item[(2)] $\mathcal{X}\setminus \mathcal{D}$ admits an $\O_{k,S}$-integral point which lies on a smooth fibre $C$ of $\pi_1$ such that $C^\circ(\O_{k,S})$ is infinite and $\pi_2|_D: D \to \P^1$ is unramified along $C \cap D$.
    \end{enumerate}
    Then $(\mathcal{X} \setminus \mathcal{D})(\O_{k,S})$ is not thin.
\end{corollary}

\begin{proof}

We first explain how condition (1b) in Corollary \ref{cor:dP} actually implies (1a).

\begin{lemma}\label{lem:noconstant}
	 Let $X$ be a del Pezzo surface of degree $d \in \{1,2,4\}$. Then a conic bundle 
	 on $X$ and its dual have intersection number at least $2$, and no singular fibres 
	 share an irreducible component.
\end{lemma}
\begin{proof}
	We use the relation $-(4/d)K_X = C + C'$ from Lemma \ref{lem:Str2conic}.
	From this $C \cdot C' = 8/d \geq 2$ and $C \neq C'$. So assume that $C$ and $C'$ are singular conics in the family which share a common component. Then
	$$-\dfrac{4}{d}K_X = 2L + L_1 + L_2$$
	where $L,L_1,L_2$ are lines on $S$ such that
	$$
    L \cdot L_1 = 1, \quad L \cdot L_2 = 1.
	$$
	However using $-K_X \cdot L = 1$ we obtain
	$$L \cdot L_1 +  L \cdot L_2 = \dfrac{4}{d} + 2$$
	which is a contradiction.
\end{proof}

Then that (1b) implies (1a) follows from the fact that $X \setminus D$ is simply connected for any smooth effective anticanonical divisor $D$ (this is a special case of \cite[Lem.~3.3.6]{Har17}).

It remains to prove Corollary \ref{cor:dP} using Theorem \ref{thm:generalisation} assuming conditions (1a) and (2). Hypotheses \eqref{hyp_intro:1}, \eqref{hyp_intro:2a}, and \eqref{hyp_intro:2c}) are immediate, and (\ref{hyp_intro:2b}) follows from Lemma~\ref{lem:conic_finiteness} since $D$ has genus $1$ for $d \leq 7$. Finally hypothesis \eqref{hyp_intro:3} follows from Lemma~\ref{lem:conic_bundle_facts}.
\end{proof}

\subsection{Proof of Theorem \ref{thm:dP}} The given conic has infinitely many integral points by Lemma~\ref{lem:pointsatinftyconic}. The result now follows from Corollary \ref{cor:dP}. \qed

\subsection{Proof of Theorem \ref{thm:cubic}}
Any line on a smooth cubic surface gives rise to a conic bundle structure by taking the residual intersections with the family of planes containing the line. We first study the geometry of the double conic bundles which arise this way, in particular the difference between the lines being skew or not.

\begin{lemma} \label{lem:cubic_lines}
    Let $X$ be a smooth cubic surface over a field $k$ which contains two distinct lines $L_1,L_2 \subset X$. Consider the associated conic bundles $\pi_1$ and $\pi_2$ with divisor classes $C_1$ and $C_2$. Let $E$ be the divisor given by the common irreducible components of both fibrations and let $D$ be a smooth hyperplane section.
    \begin{enumerate}
        \item If $L_1 \cdot L_2 = 0$, then $C_1 \cdot C_2 = 1$. Moreover $E$ consists of $5$ pairwise skew lines and $X \setminus (D \cup E)$ is not simply connected.
        \item If $L_1 \cdot L_2 = 1$, then $C_1 \cdot C_2 = 2$. Moreover $E$ consists of a single line and $X \setminus (D \cup E)$ is simply connected.
    \end{enumerate}
\end{lemma}
\begin{proof}
    The stated intersection numbers follow from the formulae 
    $$(L_1 + C_1)(L_2 + C_2) = 3, \quad L_1\cdot (L_2 + C_2) = L_2\cdot (L_1 + C_1) = 1.$$
    
    (1) It is well-known that two skew lines meet exactly 5 others; this can be seen from the fact that there are 10 lines which meet $L_1$, corresponding to the $5$ singular fibres of the conic bundle, and $L_2$ gives a section of the conic bundle, thus meets exactly 5 of these lines. From this description it is clear that they are pairwise skew. The complement is not simply connected by Lemma \ref{lem:not_simply_connected}

    (2) The divisor $E$ is given by the lines which lie in a common plane with both $L_1$ and $L_2$. As $L_1$ and $L_2$ are coplanar there is a unique such line. By Lemma \ref{lem:blow_down_fundamental group} it suffices to note that the complement of a smooth hyperplane section for a quartic del Pezzo surface is simply connected \cite[Lem.~3.3.6]{Har17}.
\end{proof}

From Lemma \ref{lem:cubic_lines}, it is now clear that in the setting of Theorem \ref{thm:cubic}, if the lines are coplanar then Corollary \ref{cor:dP} implies that the set of integral points is not thin.

It remains to show Zariski density if the lines are skew. The argument is a minor variant of the proof of Lemma \ref{lem:Zariski_dense}.  We are given a conic with infinitely many integral points. We  consider conics from the second fibration passing through these integral points. As $D$ has genus $1$, the finiteness property from Lemma \ref{lem:conic_finiteness} holds for all such conics (this does not use \eqref{eqn:not_a_section}). A topological argument as in the proof of Lemma~\ref{lem:Zariski_dense} shows that there are infinitely many such conics with infinitely many integral points. Repeating this process with the fibrations swapped gives the result. \qed

\smallskip
The above Zariski density argument applies to more general double conic bundle surfaces, but all new interesting applications appear to be to cubic surfaces (for del Pezzo surfaces of degree $1,2,4$ one can use Theorem~\ref{thm:dP}). 

\begin{question} \label{ques:cubic}
    Let $X$ be a smooth cubic surface over $\Q$ with two skew lines and smooth hyperplane section $D \subset X$, with natural  model given by taking the closure of $X$ in $\P^3_{\Z}$. If $X^\circ(\Z)$ is non-empty, then is it non-thin? Theorem \ref{thm:cubic} gives Zariski density, but our method cannot prove non-thinness, due to the corresponding complement not being simply connected (Lemma \ref{lem:cubic_lines}).
\end{question}

\begin{example} \label{ex:sums_three_cubes}
    We now use Theorem \ref{thm:cubic} to prove the integral Hilbert property for
    $$x^3 + y^3 + z^3 = 1,$$
    as originally proven in \cite[Thm.~1.15]{Coccia2019TheSurfaces}.
    Let
    $$X: \, x^3 +y^3 + z^3 + w^3 = 0,$$
    be the corresponding projective surface, with boundary divisor $D: \, w=0$.
    There are coplanar lines $L_1: x+y = z + w = 0$ and $L_2: x + z = y + w = 0.$
    We use the change variables $(x_0,x_1,x_2,x_3) = (x+y,x-y,z+w,z-w)$, to obtain
    $$X:\, x_0(x_0^2+3x_1^2)+x_2(x_2^2+3x_3^2) = 0.$$
    With this change of variables we have  $D: x_2-x_3=0$ and $L_1: x_0 = x_2=0 $.
    The conic bundle $\pi_1: X \rightarrow \P^1$ associated to $L_1$ is
    \[ \pi_1: (x_0,x_1,x_2,x_3) \mapsto \begin{cases}
        (x_0,x_2) & \text{ if }(x_0,x_2) \neq (0,0)\\
        (-x_2^2-3x_3^2,x_0^2+3x_1^2) & \text{ if }(x_0^2+3x_1^2,x_2^2+3x_3^2) \neq (0,0).
    \end{cases}
    \]
    The fibre over a point $(s,t) \in \P^1(\Q)$ is a conic $C_{s,t}$ lying in the plane $P_{s,t}$ defined by the equation $tx_0 - sx_2 = 0$.

    We next find an integer point on the surface which lies on a conic that meets the boundary divisor in a real quadratic point. Using the isomorphism $\P^2 \rightarrow P_{s,t}$, $(y_0,y_1,y_2) \mapsto (sy_0,y_1,ty_0,y_2)$, we can write
    \[ C_{s,t}: \, (s^3+t^3)y_0^2 + 3sy_1^2 + 3t y_2^2 =0 .\]
    This equation can be interpreted as lying in some $\P^2$-bundle over $\P^1$; see \cite[\S2]{Frei2018RationalSurfaces} for this perspective.
    The intersection of the conic and the boundary divisor $D: y_2-ty_0=0$ is $(s^3+4t^3)y_0^2 +3sy_1^2 = 0$. 
    For this to be real quadratic we require that $-s(s^3+4t^3)>0$ and that this is a non-square.
    The solution $(x,y,w,z)=(-9,8,1,6)$, which gives $(x_0,x_1,x_2,x_3)=(1,17,-7,5)$ satisfies this real condition, as required: it lies on the conic $C=\pi_1^{-1}(1,-7)$. 
    As for the unramified condition in Theorem \ref{thm:cubic}, the ramification points of $D$ with respect to $\pi_2$ exactly correspond to the planes through $L_2$ which meet $D$ in a non-reduced subscheme. 
    However the planes passing through the points $C_{1,-7}\cap D$ and $L_2$ are given by equations
    \[ (1-\sqrt{457})(x+z) + (13-\sqrt{457})(y+w)=0, \quad (1+\sqrt{457})(x+z) + (13+\sqrt{457})(y+w)=0. \]
    Then one checks that these planes meet $D$ in a reduced subscheme, as required. Thus Theorem \ref{thm:cubic} implies that the integral points are non-thin.
\end{example}

\section{Application to Fermat near misses} \label{sec:Fermat}

\subsection{Geometry of Fermat del Pezzo surfaces}
The aim of this section is to prove Theorem \ref{thm:Fermat}. We prepare with a study of the geometry of the surfaces
\begin{equation}\label{eq:Sn}
    X_n: \quad x^4+y^4-nz^4=w^2 \quad \subset \, \P(1,1,1,2)
\end{equation}
for $n \in \Q$. The more general class $Ax^4+By^4+Cz^4=w^2$ has been studied by Kresch and Tschinkel \cite{Kresch2004On2}, and we make use of some of their results. In particular one can find in \cite[\S 2]{Kresch2004On2} a complete description of the 56 lines on $X_n$. It follows from this that $X_n$ hax splitting field $\Q(\zeta, \sqrt[4]{n})$, where $\zeta = e^{i\pi/4}$ is a fixed choice of primitive $8$th root of unity. The Galois group of this field generically has degree $16$.

We first classify when a conic bundle exists.

\begin{proposition} \label{prop:conic_bundle_Fermat}
    Let $n \in \Q^\times$. The surface $X_n$ admits a conic bundle structure if and only if $n,2n$ or $-2n$ is a square, or $-4n$ is a fourth power.
\end{proposition}
\begin{proof}
    We begin by excluding the cases without a conic bundle structure.
    First assume that the Galois action is maximal, i.e.~that $[\Q(\zeta, \sqrt[4]{n}): \Q] = 16$. Since the Galois action on $\Pic \bar{X}_n$ is the same for any such surface, it suffices
    to demonstrate \textit{one} surface without a conic bundle structure.  
    The Magma code from \cite{Kresch2004On2} yields $\rank \Pic X_3 = 1$, so this surface has no a conic bundle structure.

    Now assume that $[\Q(\zeta, \sqrt[4]{n}): \Q]< 16$. This happens if and only if $n$ becomes a square in $\Q(\zeta)$. Since $\ker(\Q^\times/\Q^{\times 2} \to \Q(\zeta)^{\times}/\Q(\zeta)^{\times 2}) = \{ \pm 1, \pm 2\}$, this happens if and only if one of $\pm n, \pm 2 n$ is a square.

    First consider the case where $-n$ is a square. We subdivide into the cases depending on whether $-n$ becomes a square or a fourth power in $\Q(\zeta)$.
    Since $\ker(\Q^\times/\Q^{\times 4} \to \Q(\zeta)^{\times}/\Q(\zeta)^{\times 4}) 
    = \{ \pm 1, \pm 4\}$, we see that $-n$ becomes a fourth power in $\Q(\zeta)$ if and only if $-n$ or $-4n$ is a fourth power. However if $-n$ is fourth power, or $-n$ is a square but not a fourth power, there is no conic bundle structure, as follows from the fact that $\rank \Pic X_{-1} = \rank \Pic X_{-9} = 1$.

    Therefore it remains to show that a conic bundle exists in the cases in the statement of the proposition. By Lemma \ref{lem:conic_bundle_exists} it suffices to find suitable Galois invariant collections of lines on $X_n$. We briefly postpone this verification to Proposition \ref{prop:conic_bundle_equation},
    where we perform a detailed study of the conic bundle structures in question.
\end{proof}

We now write down explicitly the conic bundle structures in Proposition \ref{prop:conic_bundle_Fermat}. Let $X$ be a del Pezzo surface of degree $2$ with a conic bundle structure $X \to \P^1$. Combining with the anticanonical map $X \to \P^2, (x,y,z,w) \mapsto (x,y,z)$ yields an embedding $X \to \P^1 \times \P^2$ as a surface of bidegree $(2,2)$ (see \cite[\S5.5.2]{Frei2018RationalSurfaces} for more general descriptions of conic bundles on del Pezzo surfaces). This embedding gives a clearer way to visualise the conic bundle structure. In the statement of the next proposition, the formula for the conic bundle gives a priori only a rational map; that it extends to a morphism follows from Lemma \ref{lem:conic_bundle_exists}. We include the factorisation of the determinant of the corresponding quadratic form; this gives the locus of the singular fibres and is easy to calculate from the conic bundle equation.

\begin{proposition} \label{prop:conic_bundle_equation}
    The surfaces in $X_n$ from Proposition \ref{prop:conic_bundle_Fermat} admit the conic bundle morphism $\pi: X_n \to \P^1$ with induced image in $\P^1 \times \P^2$, given by Table \ref{tab:conic_bundles}.
    \begin{table}[htbp]
\renewcommand{\arraystretch}{1.75}
    \begin{tabular}{|c|p{13cm}|}
        \hline
        \textbf{Case} & \textbf{Conic bundle, equation in $\P^1\times \P^2$, and determinant} \\
        \hline
        \multirow{3}{*}{$n = -4m^4$} & $\pi: X_n \to \P^1, \quad  (x,y,z,w) \mapsto (x^2 + 2mxz + 2m^2z^2, w -y^2 )$\\ \cline{2-2}
                                     & \textbf{Equation:} $(s^2 - t^2)x_0^2 - 2stx_1^2 - 2m(s^2 + t^2)x_0x_2 + 2m^2(s^2 - t^2)x_2^2= 0$ \\ \cline{2-2}
                                     & \textbf{Determinant:} $-2m^{2}st(s^{2} - 2st - t^{2})(s^{2} + 2st - t^{2})$\\
        \hline

        \multirow{3}{*}{$n = m^2$} & $\pi: X_n \to \P^1, \quad (x,y,z,w) \mapsto  (x^2 - mz^2, w -y^2)$ \\ \cline{2-2}
                                   & \textbf{Equation:} $(s^2 - t^2)x_0^2 - 2stx_1^2 + m(s^2 + t^2) x_2^2= 0$\\ \cline{2-2}
                                   & \textbf{Determinant:} $-2mst(s - t)(s+t)(s^2 + t^2)$\\
        \hline

        \multirow{3}{*}{$n = 2m^2$} & $\pi: X_n \to \P^1, \quad (x,y,z,w) \mapsto  (xy-mz^2, x^2-y^2-w)$\\ \cline{2-2}
                                    & \textbf{Equation:} $ 2stx_0^2 + (2s^2-t^2)x_0x_1  - 2st x_1^2 + m(2s^2 + t^2)x_2^2 = 0$\\ \cline{2-2}
                                    & \textbf{Determinant:} $-\tfrac{m}{4}(2s^2+t^2)(4s^4+12s^2t^2+t^4)$\\
        \hline
        \multirow{3}{*}{$n = -2m^2$} & $\pi: X_n \to \P^1, \quad (x,y,z,w) \mapsto  (xy-mz^2,x^2+y^2-w)$ \\ \cline{2-2}
                                     & \textbf{Equation:} $ 2stx_0^2 - (2s^2+t^2)x_0x_1 + 2stx_1^2 +m(2s^2-t^2)x_2^2= 0$\\ \cline{2-2}
                                     & \textbf{Determinant:} $-\tfrac{m}{4}(2s^2-t^2)(2s^2-4st+t^2)(2s^2+4st+t^2)$\\
        \hline
    \end{tabular}
    \caption{Conic bundles for the surfaces $X_n$ in Proposition \ref{prop:conic_bundle_Fermat}.}
    \label{tab:conic_bundles}
\end{table}
\end{proposition}
\begin{proof}
    We achieve the descriptions by turning the proof of Lemma \ref{lem:conic_bundle_exists} into an effective method. Note  that $X_n(\Q) \neq \emptyset$ since they contain
    the rational point $(0,1,0,1)$, as necessary for Lemma \ref{lem:conic_bundle_exists}.
    The first step will be to write down Galois invariant collections of lines which meet in pairs and are otherwise disjoint.

    \textbf{Case $n = -4m^4$: } We consider the equation
$$x^4 + y^4 = w^2 - 4m^4z^4.$$
    Using the factorisation $x^4 + 4mz^4 = (x^2 + 2mxz + 2m^2z^2)(x^2 - 2mxz + 2m^2z^2)$ we find that this surface has the following Galois conjugate lines
\begin{align*}
    L: & \quad x - (1 - i)m z = 0, \quad w - y^2 = 0, \\
    L': & \quad x - (1 + i)m z = 0, \quad w - y^2 = 0,
\end{align*}
which form the singular fibre $F= L + L'$ of a conic bundle defined over $\Q$. We calculate $\H^0(X_n,F)$ by recalling that an element of this cohomology group can be represented by a rational function $f$ such that $\mathrm{div} f + F \geq 0$. This gives $\H^0(X_n,F) = \langle 1 , (w -y^2)/(x^2 + 2mxz + 2m^2z^2) \rangle$, hence the conic bundle morphism is as stated in the proposition. One easily calculates that it has the given image (we found the equation using Magma \cite{Magma}).

    \textbf{Case $n = m^2$:} We consider the equation
$$x^4 + y^4 = w^2 + m^2z^4.$$
We have the following Galois conjugate lines
\begin{align*}
    L:  & \quad  x - \sqrt{m} z = 0, \quad w - y^2 = 0, \\
    L': & \quad x + \sqrt{m} z = 0, \quad w - y^2 = 0.
\end{align*}
The calculation is then similar to the previous case and gives the stated result.

\textbf{Case $n = 2m^2$:} This case is more challenging as there is no singular fibre defined over $\Q$. It turns out that we will be able to write down lines over the field $\Q(\zeta)$ where $\zeta = e^{\pi i /4}$ is a primitive $8$th root of unity. 
We write the generators of the Galois group $\Gal(\Q(\zeta)/\Q)$ as 
\begin{equation}
    \tau: \zeta \mapsto \zeta^3, \quad \sigma: \zeta \mapsto \zeta^5.
\end{equation}
We have the equation
\[ x^4 + y^4 = w^2 +  2m^2z^4. \]
The relevant Galois orbit of lines is given by
\begin{align} \label{eqn:lines}
\begin{split}
    L: \quad & \zeta x + y = 0, \quad w + \sqrt{-2}m z^2 = 0, \\
    L^\sigma:\quad & \zeta^5 x + y = 0, \quad w - \sqrt{-2}m z^2 = 0, \\
    L^\tau:\quad & \zeta^3 x + y = 0, \quad w + \sqrt{-2} m z^2 = 0, \\    
    L^{\sigma \tau}:\quad & \zeta^7 x + y = 0, \quad w - \sqrt{-2} m z^2 = 0.
\end{split}
\end{align}
Here $F=L+L^\tau$ is a singular fibre of the conic bundle, this time defined over $\Q(\sqrt{-2})$, with conjugate singular fibre $F^\sigma = L^\sigma + L^{\sigma \tau}$. The divisor $F + F^\sigma$ is defined over $\Q$ and we need to calculate its cohomology. We first note that $\mathrm{H}^0(X_n,F) = \langle 1, f \rangle$ and $\mathrm{H}^0(X_n,F^\sigma) = \langle 1, f^\sigma \rangle$ as $\Q(\sqrt{-2})$-vector spaces, with
\[ f:= \dfrac{(\zeta^5x+y)(\zeta^7x+y)}{w+\sqrt{-2}mz^2} = \dfrac{-x^2-\sqrt{-2}xy + y^2}{w+\sqrt{-2}mz^2}, \]
where the equality holds because of $\zeta+\zeta^3 = \sqrt{-2}$ (and hence $\zeta^5 + \zeta^7 = -\sqrt{-2}$).

We have that $\mathrm{H^0}(X_n, F+F^\sigma) \otimes \Q(\sqrt{-2}) = \langle 1,f, f^\sigma \rangle$. We require a basis over $\Q$. To find this, from the equation of $X_n$ we get that $f\cdot f^\sigma=1$ and so $f^\sigma= f^{-1}$. Thus $f+f^\sigma$ and $\sqrt{-2}(f-f^\sigma)=(\zeta+\zeta^3)(f+\zeta^4f^\sigma)$ are defined over $\Q$ and linearly independent. Indeed,
\begin{align*}
f+f^\sigma &= -2\dfrac{wx^2-wy^2+2mxyz^2}{w^2+2m^2z^4}, \\
\sqrt{-2}(f-f^\sigma)&= 4\dfrac{wxy-m(x^2-y^2)z^2}{w^2 + 2m^2z^2}.
\end{align*}
Therefore the map $X_n \to \P^2$ induced by $|F+F^\sigma|$ is 
$$ (x,y,z,w) \longmapsto ( w^2 + 2m^2z^4, -2(wy^2-wx^2-2mxyz^2), 4(wxy-m(x^2-y^2)z^2)). $$
This is not the conic bundle, we want to find the map given by the divisor class $(F + F^\sigma)/2$. To do this, we note that the image in $\P^2$ is a plane conic. We want to choose an isomorphism of this conic with $\P^1$. To achieve this we use that our surfaces have a rational point given by $(0,1,0,1)$. The image of this rational point is $(1,2,0)$. The projection from this rational point is the rational map
\begin{equation}\label{eq:projection}
    \P^2 \dashrightarrow \P^1, \quad (x_0,x_1,x_2) \mapsto (2x_0 - x_1, x_2).
\end{equation}
So we compose with the map \eqref{eq:projection} and we get
\[ \begin{matrix}
    X_n & \dashrightarrow & \P^1\\
    (x,y,z,w) & \mapsto & (w^2 + 2m^2z^4 + wy^2 -wx^2 - 2mxyz^2, 2(wxy-m(x^2-y^2)z^2)).
\end{matrix} \]
This is now the conic bundle. However we would like this rational map to be defined by quadratic polynomials; we find these as follows. The quartic polynomials give elements of $\H^0(X_n, -4K_S)$. Recall from Lemma \ref{lem:Str2conic} that we have $-2K_S = C + C'$ for the dual conic bundle class $C'$, hence $-4K_S = 2C + 2C'$. The base locus of the linear system generated by the given quartic polynomials contains the lines \eqref{eqn:lines}, which are linearly equivalent to $2C$. The base locus contains a different conic with equation $xy-mz^2 = x^2-y^2-w =0$, which thus has class $C'$. Thus if we consider the linear system generated by $xy-mz^2=0$ and $x^2-y^2-w =0$, we obtain a sub linear system of $|-2K_S|$ whose base locus contains a conic with class $C'$. Using the relation $-2K_S = C + C'$ we thus find that this linear system is the same as $|C|$, as required. We asked Magma to compute the image of the induced map to $\P^1 \times \P^2$, and it gives the equation as stated in the proposition.

\textbf{Case $n = -2m^2$:} This case is very similar to the previous case so we shall we be brief. We have the equation
\[ X_n: \, x^4 +y^4 = w^2 - 2m^2 z^4\]
with Galois orbits of lines
\begin{align*}
    L &: \quad \zeta x + y = 0, \quad w + \sqrt{2}m z^2 = 0, \\
    L^\sigma&:\quad \zeta^5 x + y = 0, \quad w - \sqrt{2}m z^2 = 0, \\
    L^\tau&: \quad \zeta^3 x + y = 0, \quad w - \sqrt{2} m z^2 = 0, \\    
    L^{\sigma \tau}&:  \quad \zeta^7 x + y = 0, \quad w + \sqrt{2} m z^2 = 0.
\end{align*}
Taking $F=L+L^{\sigma\tau}$ we have $F^{\sigma} = L^{\sigma}+L^{\tau}$ and $\mathrm{H}^0(X_n,F) = \langle 1, f \rangle$ and $\mathrm{H}^0(X_n, F^\sigma) = \langle 1, f^\sigma \rangle$ as a $\Q(\sqrt{2})$--vector space, where
\[ f = \dfrac{(\zeta^5 x + y)(\zeta^3x + y)}{w-\sqrt{2}mz^2} = \dfrac{x^2 - \sqrt{2}xy + y^2}{w - \sqrt{2}mz^2}, \]
\[ f^\sigma = \dfrac{(\zeta x+ y)(\zeta^{7} x+y)}{w+\sqrt{2}mz^2} = \dfrac{x^2 + \sqrt{2}xy + y^2}{w + \sqrt{2}mz^2}.\]
Again we get that $f\cdot f^\sigma = 1$ from the equation of $X_n$ and so $f^\sigma = f^{-1}$. Thus
\begin{align*}
f+ f^\sigma &= \dfrac{2(wx^2 + wy^2 -2mxyz^2)}{w^2-2m^2z^4}, \\
\sqrt{2}(f^\sigma-f) & = \dfrac{4(m(x^2+y^2)z^2-wxy)}{w^2-2m^2z^4}
\end{align*}
are defined over $\Q$ and we find that the map associated to $F+F^\sigma$ is
$$ (x,y,z,w) \longmapsto ( w^2 - 2m^2z^4, 2(wx^2+ wy^2 -2mxyz^2 ), 4(m(x^2 + y^2)z^2 - wxy)). $$
Again composing with the projection from the same rational point, we get
\[ \begin{matrix}
    X_n & \dashrightarrow & \P^1\\
    (x,y,z,w) & \mapsto & (w^2-2m^2z^4 + 2mxyz^2 - wx^2 - wy^2, 2(m(x^2 + y^2)z^2 - wxy)).
\end{matrix} \]
The base locus of this contains the conic $xy-mz^2 =x^2+y^2-w = 0$, thus we obtain the map and image of the induced map to $\P^1 \times \P^2$ stated in the proposition via a Magma calculation.
\end{proof}

\begin{remark} \label{rem:non_rational}
    One can show using the formulae for the Picard rank of a conic bundle \cite[Lem.~2.1]{Frei2018RationalSurfaces} that if one of $n,2n,-2n$ is a square, but $n$ is not a fourth power, then $\rank \Pic X_n = 2$. It follows by a theorem of Iskovskih \cite[Thm.~4]{Iskovskih1980MinimalFields} that such surfaces are minimal, in particular non-rational. Therefore there is no way to prove Theorem \ref{thm:Fermat} by performing a birational transformation to a del Pezzo surface of larger degree.
\end{remark}

\subsection{Proof of Theorem \ref{thm:Fermat}}
By Theorem \ref{thm:dP}, it suffices to show that our conditions correspond to an integral point which lies on a smooth conic that meets the boundary divisor in a real quadratic point and such that the dual conic bundle is unramified. We use Proposition \ref{prop:conic_bundle_equation}, which gives explicit descriptions for the conic bundle structures for the surfaces in Theorem \ref{thm:Fermat}, which we denote by $\pi_1$ with dual conic bundle $\pi_2$ obtained via the Geiser involution. For $n=1, -4$ one checks that the points given in Table \ref{tab:solutions} have this property. So we assume for the remainder of the proof that we are not in any of these cases, up to a fourth power. With respect to the conic bundle model, the boundary divisor is given by $D: x_2 = 0$. From Table~\ref{tab:conic_bundles}, one easily deduces that the discriminant of $\pi_1|_D :D \to \P^1$ is given as:

\begin{table}[htb] 
    \centering
    \begin{tabular}{c|c|c|c} 
        $n$ & $m^2$ & $2m^2$ & $-2m^2$   \\  \hline
         & $8st(s - t)(s+t)$ & $4s^4+12s^2t^2+t^4$ & $(2s^2-4st+t^2)(2s^2+4st+t^2)$
    \end{tabular}
    \caption{Discriminant of $\pi_1|_D :D \to \P^1$}
    \label{tab:discriminant}
\end{table}

\begin{lemma} \hfill \label{lem:Fermat}
	\begin{enumerate}
		\item Any integer solution lies on a smooth fibre.
		\item The only smooth fibres which meet the boundary divisor in at least one rational point are for $n = \pm 2m^2$ and $s= 0$ or $t = 0$.
        \item For every smooth fibre $C$ for $\pi_1$, the map $\pi_2|_D :D \to \P^1$ is unramified along $C \cap D$.
	\end{enumerate}
\end{lemma}
\begin{proof}
 (1) For the cases $n = \pm 2m^2$, this is because there is no singular fibre over $\Q$. For the case $n = m^2$, the description from Proposition \ref{prop:conic_bundle_equation} reveals that there are $4$ singular fibres over $\Q$, and that these singular fibres are all non-split. It follows that the only rational point on a singular fibre is the intersection of two lines. However none of these define integral points since they all have $x_2 = 0$.
 
 (2) The boundary divisor is the elliptic curve $D: x^4 + y^4 = w^2$, whose only rational points are $(1:0:\pm1)$ and $(0:1:\pm 1)$. This corresponds to the boundary divisor $x_2 = 0$ in the conic bundle model. Therefore it suffices to find the $4$ rational points on this curve and see where they lie. When $n = m^2$ all these rational points are the singular points of the singular fibres, so do not lie on a smooth fibre. For $n = \pm 2m^2$ the $4$ rational points occur when $s = 0$ or $t = 0$, as required.

 (3) Our boundary divisor $D$ has the special property that it is preserved under the action of the Geiser involution $\iota: w \mapsto -w$. One checks that $\iota$ acts on $D$ via:
 \begin{table}[htb] 
    \centering
    \begin{tabular}{c|c|c|c} 
        $n$ & $m^2$ & $2m^2$ & $-2m^2$   \\  \hline
        $(s,t) \mapsto $ & $(t,-s)$ & $(t,-2s)$ & $(t,2s)$
    \end{tabular}
  \end{table} 

  Let $R_{i} \subset D$ denote the ramification locus of $\pi_i|_D: D \to \P^1$, viewed as a closed subscheme of degree $4$; we have $\iota(R_1) = R_2$. 
  In each case one checks that the action preserves the discriminant of $\pi_1|_D$ as in Table \ref{tab:discriminant}. It follows that $\iota$ preserves $R_1$, thus $R_1 = R_2$. However it also follows from Tables \ref{tab:conic_bundles} and \ref{tab:discriminant} that $D$ ramifies only above the singular fibres of $\pi_1$, thus the same holds for $\pi_2$ since $\iota$ preserves intersection multiplicities and maps singular fibres to singular fibres.
  
Now assume for a contradiction that $C$ is a smooth fibre for $\pi_1$ such that $C \cap R_2 \neq \emptyset$. It follows that $\iota(C) \cap R_1 \neq \emptyset$. However from the above we know that it is only the singular fibres of $\pi_2$ which pass through $R_1 = R_2$. It follows that $\iota(C)$ is singular, thus $C$ is singular, which is a contradiction.
\end{proof}

Lemma \ref{lem:Fermat} shows that, away from the fibres over $st=0$ in the case $n = \pm 2m^2$, any integral point lies on a smooth conic which meets the boundary divisor in a closed point of degree two which is an unramified point with respect to $\pi_2$. It remains to verify that the real conditions in Theorem~\ref{thm:Fermat} are now exactly asking that the residue field of this point is real quadratic.
However this is now clear from Table~\ref{tab:discriminant}, noting that for $n = 2m^2$ the discriminant $4s^4 + 12s^2t^2 + t^4$ is positive definite, and for $n = -2m^2$ the discriminant expands to $4s^4 -12 s^2t^2 + t^4$. \qed 

\begin{remark}
    By Lemma \ref{lem:Fermat}, the trivial solutions presented in Remark \ref{rem:trivial} exactly correspond to integral points which lie on a conic that meets the divisor $D$ in $2$ rational points. Such conics have only finitely many integral points (Lemma \ref{lem:pointsatinftyconic}), so are useless for our method.
\end{remark}

\bibliographystyle{alpha}
\bibliography{references}
\end{document}